\documentclass[a4paper,10pt]{amsart}
\usepackage[usenames]{color} 
\usepackage{amssymb} 
\usepackage{amsmath} 
\usepackage{amsthm}
\usepackage{mathrsfs}
\usepackage[all]{xy}
\usepackage[utf8]{inputenc} 
\newtheorem{theo}{Theorem}[section]
\newtheorem{lem}{Lemma}[section]

\newtheorem{prop}{Proposition}[section]
\newtheorem{cor}{Corollary}[section]
\newtheorem{rem}{Remark}[section]
\theoremstyle {remark}

\begin{document}
\title{Two minimal unique ergodic diffeomorphisms on a manifolds and their smooth crossed product algebras }
\author{Hongzhi Liu\\
Email: 1063733099@qq.com}

\thanks{Supported by NNSF of China (11201171).}

\maketitle 
\pagestyle{plain}
\begin{abstract}

In this article we construct two minimal unique ergodic diffeomorphisms $\alpha$ and $\beta$ on $S^3 \times S^{6} \times S^{8} $. We will show that $C(S^3 \times S^{6} \times S^{8}) \rtimes_\alpha \mathbb{Z} $ and $C(S^3 \times S^{6} \times S^{8})\rtimes_\beta \mathbb{Z} $ are equivalent to each other, while $C^\infty (S^3 \times S^{6} \times S^{8})\rtimes_\alpha \mathbb{Z} $ and $C^\infty(S^3 \times S^{6} \times S^{8} )\rtimes_\beta \mathbb{Z} $ are not.



\end{abstract}

\section{introduction}

Isomorphism between two irrational rotation algebras implies flip equivalence between their corresponding irrational rotation transformations (\cite{pv}, \cite{rie}). Giordano, Putnam and Skau have given a classification of dynamical systems on Cantor set based on $C^*$-crossed product algebras (\cite{gps}). 

However, different dynamical systems may give equivalent $C^*$-crossed product algebras (see \cite{chis}, \cite{ch1} for examples). It is interesting to investigate smooth crossed product algebra in these cases. Let $g$ and $h$ be minimal unique ergodic diffeomorphisms of $S^{2m+1}$ and $S^{2n+1}$ respectively, with $m,n\geq 1$. In \cite{ch5}, N. C. Phillips proved that $C(S^{2m+1})\rtimes_g \mathbb{Z} \cong C(S^{2n+1})\rtimes_h \mathbb{Z}$. In \cite{liu}, the author proved that $C^\infty(S^{2m+1})\rtimes_g \mathbb{Z} \ncong  C^\infty(S^{2n+1})\rtimes_h \mathbb{Z}$ by checking the grading structure of cyclic cohomology. People may argue that this is too obvious since these two diffeomorphisms are of different manifolds. In this article we construct two diffeomorphisms of a same manifold giving same $C^*$-crossed product algebra and different smooth crossed product algebras. 

We introduce several notions and theories important for our construction in the next section. In the third section we give the construction (Theorem \ref{cons}) of minimal unique ergodic diffeomorphisms $\alpha$ and $\beta$ of $C(S^3 \times S^{6} \times S^{8})$. The fourth and fifth section are devoted to computation of $K$-theory and cyclic cohomology. These computation together show that 
\begin{eqnarray*}
C(S^3 \times S^{6} \times S^{8})\rtimes_{\alpha} \mathbb{Z} &\cong &C(S^3 \times S^{6} \times S^{8})\rtimes_{\beta} \mathbb{Z},\\
C^\infty(S^3 \times S^{6} \times S^{8})\rtimes_{\alpha} \mathbb{Z} &\ncong & C^\infty(S^3 \times S^{6} \times S^{8})\rtimes_{\beta} \mathbb{Z}.
\end{eqnarray*}
In the last section we prove Theorem \ref{cons}. 

\section{Preliminary}

Let $M$ be a finite dimensional compact manifold. Choose finitely many vector fields $X_1$, $X_2$, $\dots$, $X_n$ on $M$ which can span the tangent space at any point ($n$ is not necessarily equal to the dimension of $M$). Define seminorms $\| \bullet \|_n$:
 \[\|f\|_n=\sum_{1\leq k_1\leq \dots \leq k_n\leq n}\|X_{k_n}X_{k_{n-1}}\dots X_{k_1}f\|_\infty, n\in \mathbb{Z}_+\cup \{0\}, f\in C^\infty (M).\] 
Let $\alpha$ be a minimal unique ergodic diffeomorphism of $M$. Let $C^\infty (M)_\alpha[u,u^{-1}]$ be the algebraic crossed product of $C^\infty (M)$ by $\mathbb{Z}$. Let $\|\alpha^t\|_i$ be the operator seminorms defined by $\|\bullet\|_i$ on $C^\infty (M)$, i.e.  
\[\|\alpha^t\|_i \triangleq \sup_{f\in C^\infty(M), \|f\|_i =1} \|\alpha^t(f)\|_i.\]
These seminorms defines topology $\mathcal{T}$ on $C^\infty (M)$.

Define a sequence of maps 
\[\rho_k:\mathbb{Z} \to  \mathbb{R}^+, k=1,2,\dots ,\]
by 
$\rho_k(n)  = \nolinebreak \sup_{i\leqslant k}(\sum_{t=-n}^n\|\alpha^t\|_i)^k.$

Endow $C^\infty (M)_\alpha[u,u^{-1}]$ with the topology $\mathcal{T}$ defined by the following  seminorms
\[ \|\sum_n f_n u^n\|_k=\sup_n \rho_k(n)\|f_n\|_k, f_n \in C^\infty (M).\]
This topology does not depend on the choice of $X_1$, $X_2$,  $\dots$, $X_n$. Then the completion of $C^\infty (M)_\alpha[u,u^{-1}]$ is the smooth crossed product algebra $C^\infty(M)\rtimes_\alpha \mathbb{Z}$. This is a locally convex topology  algebra.

Alain Connes defines cyclic cohomology in \cite{Al1}. For locally convex topology  algebra $\mathcal{A}$, let $HC^i(\mathcal{A})$ be the cyclic cohomology. There are two inductive limits 
\[ HC^0(\mathbb{C})\dots \overset{S}\to HC^{2n}(\mathbb{C}) \overset{S}\to  HC^{2n+2}(\mathbb{C})\overset{S}\to \dots,\]
\[ HC^1(\mathbb{C})\dots \overset{S}\to HC^{2n+1}(\mathbb{C}) \overset{S}\to  HC^{2n+3}(\mathbb{C})\overset{S}\to \dots.\]
The limit groups are the so called periodic cyclic cohomology
\[HP^i (\mathcal{A}) \triangleq \lim_\to HC^{2n+i}(\mathcal{A}),i=0,1.\]
Let $S(HC^*(\mathcal{A}))\subset HP^*(\mathcal{A})$ be the image of $HC^*(\mathcal{A})$ in $HP^*(\mathcal{A})$. Groups $S(HC^n(\mathcal{A}))/S(HC^{n-2}(\mathcal{A}))$ are the grading structure of cyclic cohomology of $\mathcal{A}$. 

Grading groups can be computed as $E_\infty^n(\mathcal{A})$ (see \cite{Al1} and \cite{nest}, or \cite{liu} for details). These groups are also called as deRham homology of $\mathcal{A}$ \cite{nestt}.

Let us recall some facts about minimal and unique ergodic diffeomorphism. 
 
\begin{prop}[\cite{fathi}]\label{mind}
Let $M$ be a compact manifold, $G$ be a group acting on $M$, then the following two conditions are equivalent:\\
$\mathnormal{1)} $ G action is minimal. \\
$\mathnormal{2)} $ There is a finite subsets $\{g_1, \dots , g_n \} \subset G$, such that for any non-empty open sets $U$ in $M$, there is $M= \cup_{i=1}^n g_i(U) $.
\end{prop}

\begin{prop}[\cite{fur}]\label{ergd}
 Let $M$ be a compact manifold, $f$ be a diffeomorphism on $M$. Then the followings are equivalent: \\
$\mathnormal{1)} $ $f$ is unique ergodic. \\
$\mathnormal{2)} $ For any $\phi \in C(M)$, $\frac{1}{n} \sum^{n-1}_{i=0} \phi \circ f^i$ converge uniformly.
\end{prop}

Let $m$ be the set of minimal diffeomorphisms on a compact manifold $M$, $u_e$ be the set of unique ergodic on $M$. Let $D(M)$ be the set of all the diffeomorphisms on $M$. Endow this set with the topology $\mathcal{T}$. 

For any open set $U$, $\phi \in C(M)$, $\epsilon > 0 $, define open sets $\iota (U) \in D(M)$ and $\iota (\phi, \epsilon )\in D(M)$ as 
\[\iota(U)=\{f\in D(M) | \exists k\in \mathbb{N}, s.t. U\cup \dots f^k(U)=M \},\]
\[\iota (\phi, \epsilon ) =\{ f \in D(M) | \exists n\in \mathbb{N},  r \in \mathbb{R}, s.t. \|\frac{1}{n} \sum^{n-1}_{i=0} \phi\circ f^i - r \|< \epsilon \}.\]
Let $(U_i)_{i\in \mathbb{N}}$ be a topological base of topology $\mathcal{T}$. $\{\phi_l\}_{l\in\mathbb{N}}$ are dense in $C(M)$. By Propositions \ref{mind} and \ref{ergd}, we have 
\[m= \cap_i \iota(U_i), \]
\[u_e = \cap_l \cap_{k\geq1} \iota(\phi_l, 1/k ),\]
where $k\in \mathbb{N}$.

\section{The construction}

Recall that $S^1=\{\|a_1\|=1 | a_1\in \mathbb{C}\}$, $S^3=\{\|a_2\|+\|a_3\|=1 | a_2, a_3\in \mathbb{C}\}$. There exist a free $S^1$ actions on $S^3$ defined by
\[\mathcal{R}(a_1,a_2,a_3 ) = (a_1a_2, a_1a_3)\] 
is a free $S^1$ action. On manifold $S^3 \times S^{6} \times S^{8}$, we can also define a free $S^1$ action as $R\otimes Id \otimes Id$, which will be denoted by $R$ in the following.

Let $\mathcal{R}_t(x)$ be $\mathcal{R}(t, x)$, for any $t \in S^1$ and $x \in S^3 \times S^{6} \times S^{8}$. Let $\mathcal{P}_6$  and $\mathcal{P}_8$ be the antipode map on $S^6$ and $S^8$ respectively. Consider two sets of diffeomorphisms of $S^3 \times S^{6} \times S^{8}$
\[\Omega(\mathcal{R}, \mathcal{P}_6)=\{g\circ \mathcal{R}_t \circ \mathcal{P}_6 \circ g^{-1}  | t \in S^1, g \in D(S^3 \times S^{6} \times S^{8})\},\]
\[\Omega(\mathcal{R}, \mathcal{P}_8)=\{g\circ \mathcal{R}_t \circ \mathcal{P}_8 \circ g^{-1}  | t \in S^1, g \in D(S^3 \times S^{6} \times S^{8})\}.\]

We claim

\begin{theo} \label{cons}
There are minimal unique ergodic diffeomorphisms in both $\overline{\Omega(\mathcal{R}, \mathcal{P}_6)}$ and $\overline{\Omega(\mathcal{R}, \mathcal{P}_8)}$. 
\end{theo}

We will postpone our proof to the last section and show what this theorem implies first. In the following, $\alpha$ and $\beta$ are  minimal unique ergodic diffeomorphisms in $\overline{\Omega(\mathcal{R}, \mathcal{P}_6)}$ and $\overline{\Omega(\mathcal{R}, \mathcal{P}_8)}$ respectively. 

\section{ Classify $C^*$-algebras }

Recall the classification theory in \cite{win}

\begin{theo}[\cite{win}]\label{ccla}
Let $\mathcal{C}$ be the class of $C^*$-algebras having the following properties:

$(\mathnormal{1}) $ every $\mathit{A} \in \mathcal{C} $ has the form $C(M)\rtimes_{\gamma} \mathbb{Z} $ for some infinite, compact, finite dimensional, metrizable space and minimal homeomorphism $\gamma : M\to M$.

$(\mathnormal{2}) $	the projections of every $\mathit{A} \in \mathcal{C}$ separate traces.

If $\mathit{A,B}\in \mathcal{C}$ and there is a graded ordered isomorphism $\phi : K_*(\mathit{A})\to K_*(\mathit{B})$, then there is a $*$-isomorphism $\Phi:\mathit{A}\to \mathit{B}$ which induces $\phi $.
\end{theo}

The second condition can be removed in case of unique ergodicity. Let's compute graded $K$ theory of $C(S^3 \times S^{6} \times S^{8})\rtimes_{\alpha} \mathbb{Z}$ and $C(S^3 \times S^{6} \times S^{8})\rtimes_{\beta} \mathbb{Z}$. Note that there are a sequence of $g_n, h_n \in D(S^3 \times S^{6} \times S^{8}), t_n, s_n \in S^1$ such that 
\[\lim_{n\to \infty} g_n \mathcal{R}_{t_n} \mathcal{P}_6 g_n^{-1}=\alpha, \lim_{n\to \infty} h_n \mathcal{R}_{s_n} \mathcal{P}_8 h_n^{-1}=\beta. \]

First, we compute $K_0$ and $K_1$ group of $C(S^3 \times S^{6} \times S^{8})\rtimes_{\alpha} \mathbb{Z}$. Recall the Pimsner-voiculescu six exact sequence (\cite{pv}):
\begin{displaymath} 
 \xymatrix
 {K_0 (C(S^3 \times S^{6} \times S^{8})) \ar[r] &K_0(C(S^3 \times S^{6} \times S^{8})\rtimes_{\alpha}\mathbb{Z}) \ar[r] & K_1(C(S^3 \times S^{6} \times S^{8})) \ar[d]^{1-\alpha} \\
  K_0 (C(S^3 \times S^{6} \times S^{8})) \ar[u]^{1-\alpha}  & K_1(C(S^3 \times S^{6} \times S^{8}) \rtimes_{\alpha}\mathbb{Z} )\ar[l] & K_1(C(S^3 \times S^{6} \times S^{8})) \ar[l].}
 \end{displaymath} 

The $K$ theory of $C(S^3 \times S^{6} \times S^{8})$ are 
\begin{eqnarray*}
K_0(C(S^3 \times S^{6} \times S^{8})) &\cong &\mathbb{Z}\oplus \mathbb{Z}\oplus \mathbb{Z}\oplus \mathbb{Z},\\
K_1(C(S^3 \times S^{6} \times S^{8})) &\cong &\mathbb{Z}\oplus \mathbb{Z}\oplus \mathbb{Z}\oplus \mathbb{Z},
 \end{eqnarray*}

$\alpha$ induces $K$-theoretic map on both $K_0$ and $K_1$ as
\begin{displaymath}
	\left(\begin{array}{cccc}
	1&0&0&0\\
	0&-1&0&0\\
    0&0&1&0\\
    0&0&0&-1 
	\end{array}
	\right).
\end{displaymath}
It is immediately to see that 
\begin{eqnarray*}
K_0(C(S^3 \times S^{6} \times S^{8})\rtimes_{\alpha}\mathbb{Z}) &\cong &\mathbb{Z}\oplus \mathbb{Z}\oplus \mathbb{Z}\oplus \mathbb{Z},\\
K_1(C(S^3 \times S^{6} \times S^{8})\rtimes_{\alpha}\mathbb{Z}) &\cong &\mathbb{Z}\oplus \mathbb{Z}\oplus \mathbb{Z}\oplus \mathbb{Z}.
\end{eqnarray*}

Note that $H^1(S^3 \times S^{6} \times S^{8})=0$. 
\begin{theo}[Corollary 3, \cite{AL3}]\label{zran}
Let $M$ be a compact smooth manifold, $H^1(M)=0$, $\phi$ a minimal diffeomorphism of $M$, then $C(M)\rtimes_\phi \mathbb{Z}$ is a simple $C^*$-algebras without nontrivial idempotent. 
\end{theo}

Alain Connes actually had proven that for $C^*$-algebra mentioned in Theorem \ref{zran},  the range of traces is $\mathbb{Z}$. $C^*$-algebra $A$ satisfies $K$-theoretic version of Blackadar's Second Fundermental Comparability question if for any $\eta\in K_0(A)$, $\tau(\eta)>0$ for any trace $\tau $ implies that there is a projection $P\in A$ such that $[P]=\eta$. N. C. Phillips proved that $C(S^3 \times S^{6} \times S^{8})\rtimes_{\alpha} \mathbb{Z}$ satisfies $K$-theoretic version of Blackadar's Second Fundermental Comparability question (\cite{ch5}). Hence we can conclude that $K_0(S^3 \times S^{6} \times S^{8})\rtimes_{\alpha}\mathbb{Z})\cong \mathbb{Z}_+\oplus \mathbb{Z}\oplus \mathbb{Z}\oplus \mathbb{Z}$.
No difference would be made if one substitutes $\alpha$ by $\beta$. 

Theorem \ref{ccla} then implies that
\[C(S^3 \times S^{6} \times S^{8})\rtimes_{\alpha}\mathbb{Z} \cong C(S^3 \times S^{6} \times S^{8})\rtimes_{\beta}\mathbb{Z}.\]

\section{cyclic cohomology}

We prove that $C^\infty (S^3 \times S^{6} \times S^{8})\rtimes_\alpha \mathbb{Z} \ncong C^\infty (S^3 \times S^{6} \times S^{8})\rtimes_\beta \mathbb{Z} $ in this section. For this purpose we will compute the grading structure of their periodical cyclic cohomology.

Let $M$ be a compact manifold and $\gamma $ a diffeomorphism of it. Recall Nest's six exact sequence:

\begin{theo} [\cite{nest}] \label{csix} 
\begin{displaymath} 
 \xymatrix
 {HP^{ev} (C^\infty (M)) \ar[r] & HP^{od}(C^\infty(M)) \rtimes_{\gamma}\mathbb{Z} ) \ar[r]& HP^{od}(C^\infty(M)) \ar[d]^{1-\gamma} \\
  HP^{ev} (C^\infty (M)) \ar[u]^{1-\gamma}  & HP^{ev}(C^\infty(M) \rtimes_{\gamma}\mathbb{Z} )\ar[l] & HP^{od}(C^\infty(M)) \ar[l] }.
 \end{displaymath}
\end{theo}
 
Let $H_*(M)$ be the deRham homology of $M$. 

\begin{theo}[\cite{Al1}]\label{cman}
\[HP^{ev}(M)\cong H_0(M)\oplus \dots \oplus H_{2n}\oplus \dots,\]
\[HP^{ev}(M)\cong H_1(M)\oplus \dots \oplus H_{2n+1}\oplus \dots,\]
\end{theo}
By basic deRham homology and cohomology theory, for a $m$ dimensional manifold, we have $H_n\cong H^{m-n}$. 

Now we have got all the recipe prepared to compute the periodic cyclic cohomology. For $\alpha$, it induces maps $\alpha_*$ between the $H_*$ as follows:
\[\alpha_0=1, \alpha_6=-1, \alpha_8=1, \alpha_{14}=-1,\]
\[\alpha_3=1, \alpha_9=-1, \alpha_{11}=1, \alpha_{17}=-1,\]
while $\beta_*$ are
\[\beta_0=1, \beta_6=1, \beta=-1, \beta_{14}=-1,\]
\[\beta_3=1, \beta_9=1, \beta_{11}=-1, \beta_{17}=-1.\]
Hence the $ HP^*(C^\infty(S^3 \times S^{6} \times S^{8}))_{\alpha}\mathbb{Z})$ and $ HP^*(C^\infty(S^3 \times S^{6} \times S^{8})) \rtimes_{\beta}\mathbb{Z})$ can be computed as 
\begin{eqnarray*}
HP^{od}(C^\infty(S^3 \times S^{6} \times S^{8}))\rtimes_{\alpha}\mathbb{Z})&=& 	\mathbb{C}\oplus \mathbb{C}\oplus\mathbb{C}\oplus\mathbb{C},\\
HP^{ev}(C^\infty(S^3 \times S^{6} \times S^{8}))\rtimes_{\alpha}\mathbb{Z})&=& 	\mathbb{C}\oplus \mathbb{C}\oplus\mathbb{C}\oplus\mathbb{C},\\
HP^{od}(C^\infty(S^3 \times S^{6} \times S^{8}))\rtimes_{\beta}\mathbb{Z})&=& 	\mathbb{C}\oplus \mathbb{C}\oplus\mathbb{C}\oplus\mathbb{C},\\
HP^{ev}(C^\infty(S^3 \times S^{6} \times S^{8}))\rtimes_{\beta}\mathbb{Z})&=& 	\mathbb{C}\oplus \mathbb{C}\oplus\mathbb{C}\oplus\mathbb{C}.
\end{eqnarray*}

Now let us recall how to compute groups $E_\infty^n$. Let $\Psi_n$ be the space of $\mathit{n}$-th deRham currents of $M$, $\partial$ be the usual boundary map, 
\begin{eqnarray*}
H^n_{eq}(M, \gamma) & \triangleq &homology\: group\: of\: (Ker(1-\gamma) |\Psi_n , \partial),\\
H^n_{coeq}(M, \gamma) &\triangleq & homology\: group\: of\:(Coker(1-\gamma) | \Psi_n, \partial).
\end{eqnarray*}

\begin{theo}[\cite{nest}] \label{cgra} 
$E_\infty ^n(C^{\infty}(M) \rtimes_\alpha\mathbb{Z}) = H^n _{eq}(M,\gamma) \oplus H^{n-1} _{coeq}(M,\gamma)$.
\end{theo}

Therefore computation of the grading structure can be reduced to computation of $H^n_{eq}(M, \gamma)$ and $H^n_{coeq}(M, \gamma)$. 

\begin{prop}\label{6gra}
	\[H^0_{coeq}(S^3 \times S^{6} \times S^{8}, \beta)=\mathbb{C},  H^6_{coeq}(S^3 \times S^{6} \times S^{8}, \beta)=\mathbb{C},\]
	\[H^3_{eq}(S^3 \times S^{6} \times S^{8},\beta)=\mathbb{C}, H^9_{eq}(S^3 \times S^{6} \times S^{8}, \beta)=\mathbb{C},\]
	with all the other $H^*_{eq}(S^3 \times S^{6} \times S^{8}, \beta)$ and $H^*_{coeq}(S^3 \times S^{6} \times S^{8}, \beta)$ nothing but $\{0\}$. Further more,  we can conclude that 
\[
\left\{	
\begin{array}{cc}
E_\infty^i(C^{\infty}(S^3 \times S^{6} \times S^{8}) \rtimes_\beta\mathbb{Z})\cong \mathbb{C}, & i= 1,3,7,9\\
E_\infty^i(S^3 \times S^{6} \times S^{8}) \rtimes_\beta \mathbb{Z})\cong \{0\},& i=\ else .
	\end{array}
\right.
\]
\end{prop}
\begin{proof}
 In this proof, we denote the volume form on $S^3$, $S^6$, $S^8$ by 
 $dvol_3$, $dvol_6$, $dvol_8$ respectively. \\
	$\mathnormal{(1)}$ $H^0_{coeq}(S^3 \times S^{6} \times S^{8}, \beta)$. $dvol_3\wedge dvol_6 \wedge dvol_8$ defines an element $\tau$ in $\Psi_0$. For any $f\in C^\infty (S^3 \times S^{6} \times S^{8})$, 
	\[\tau(f)\triangleq \int_{S^3 \times S^{6} \times S^{8}}f dvol_3\wedge dvol_6 \wedge dvol_8.\]
	By stokes theorem, $\partial \tau=0$. However, there are no $\tau_1 \in \Psi_0$ and $\tau_0 \in \Psi_0$ such that
	\[\tau= \partial\tau_1 + (1-\beta)\tau_0. \]
	In fact, for $f\in C^\infty(S^3 \times S^{6} \times S^{8}), f\equiv 1$, $\tau(f)=vol(S^3)vol(S^6)vol(S^8)\neq 0$, but $ \partial\tau_1(f) + (1-\beta)\tau_0(f)=0$. Hence $H^0_{coeq}(S^3 \times S^{6} \times S^{8}, \beta)$ contains at least one direct summand of $\mathbb{C}$.\\ 
	$\mathnormal{(2)}$ $H^3_{eq}(S^3 \times S^{6} \times S^{8}, \beta)$.  $ dvol_6 \times dvol_8$ defines an element $\tau$ in $\Psi_3$ as follows, for any $3$-form $\omega \in \Lambda^3$,
	\[\tau(\omega)=  \int_{S^3 \times S^{6} \times S^{8}} \omega\wedge dvol_6 \wedge dvol_8.\]
	$\tau $ will never be a boundary since $\tau(dvol_3)= vol(S^3)vol(S^6)vol(S^8)\neq 0$. However we have $\partial (\tau)=0$. In fact, one need only to concern the boundary 
	$\omega \in \partial\Lambda^2_{S^3}$. This implies that $H^3_{eq}(S^3 \times S^{6} \times S^{8}, \beta)$ contains at least one direct summand of $\mathbb{C}$.\\
	$\mathnormal{(3)}$ By the same argument as shown in $\mathnormal{(2)}$ we obtain that 
	$H^9_{eq}(S^3 \times S^{6} \times S^{8},\beta)$ contains at least one direct summand of $\mathbb{C}$, which is generated by 
	$dvol_8$.\\
	$\mathnormal{(4)}$ $H^6_{coeq}(S^3 \times S^{6} \times S^{8}, \beta)$. Consider the element $\tau$ defined by 
	\[\tau(\omega)=\int_{S^3\times S^6\times S^8} \omega\wedge dvol_3 \wedge dvol_8, \forall \omega\in \wedge^6. \]
	 As we have shown, $\partial (\tau)=0$. We claim that there are no $\tau_7 \in \Psi_7$ and $\tau_6 \in \Psi_6$ such that
	\[\tau= \partial\tau_7 + (1-\beta)\tau_6. \]
	Recall that there is a sequence of $g_n \in D(S^3 \times S^{6} \times S^{8})$, and $s_n \in S^1$ such that 
	\[\lim_{n\to \infty} g_n \mathcal{R}_{s_n} \mathcal{P}_8 g_n^{-1} =\beta.\]
	For any $n$, $g_n$ induce either $I$ or $-I$ on $H^6(S^3 \times S^{6} \times S^{8})$. Without loss of generality, we assume that for all $n$, $g_n$ induces identity map on $H^6(S^3 \times S^{6} \times S^{8})$.  Assume that there exist $\tau_7 \in \Psi_7$ and $\tau_6 \in \Psi_6$, we have 
\begin{align*}
&(1-\beta)\tau_6(dvol_6 -g_n dvol_6)\\
=&\tau(dvol_6 -g_n dvol_6)-\partial \tau_7(dvol_6 -g_n dvol_6)\\
=&0
\end{align*} 
	 Note that $g_n \mathcal{R}_{s_n} \mathcal{P}_8 g_n^{-1}$ converge to $\alpha$ uniformly, 
\begin{eqnarray*}
\tau((1- \beta)dvol_6 )&=& \tau((1- \beta)g_n dvol_6)\\
&=& \tau(\lim_{n\to \infty} (1- g_n \mathcal{R}_{s_n} \mathcal{P}_8 g_n^{-1}) g_n dvol_6)\\
	&=&\tau(\lim_{n\to \infty} g_n(1- \mathcal{R}_{s_n} \mathcal{P}_8)  dvol_6)
	\\
	&=&0.
\end{eqnarray*}
Thus $\partial \tau_7(dvol_6) + (1-\beta)\tau_6(dvol_6)=0$. However, $\tau(dvol_6)$ is not $0$. Hence there contains at least one direct summand of $\mathbb{C}$ in $H^6_{coeq}(S^3 \times S^{6} \times S^{8}, \beta)$.
In summary, each of $E_\infty ^1(C^{\infty}(S^3 \times S^{6} \times S^{8}) \rtimes_\beta \mathbb{Z})$, $E_\infty ^3(C^{\infty}(S^3 \times S^{6} \times S^{8}) \rtimes_\beta \mathbb{Z})$, $E_\infty ^7(C^{\infty}(S^3 \times S^{6} \times S^{8}) \rtimes_\beta\mathbb{Z})$,$E_\infty ^9(C^{\infty}(S^3 \times S^{6} \times S^{8}) \rtimes_\beta \mathbb{Z})$ contains at least one direct summand of $\mathbb{C}$. Recall that 
\[HP^{od}(C^\infty(S^3 \times S^{6} \times S^{8})) \rtimes_{\beta}\mathbb{Z})= \mathbb{C}\oplus \mathbb{C}\oplus\mathbb{C}\oplus\mathbb{C},\]
one can see 

\[
\left\{	
\begin{array}{cc}
E_\infty^i(C^{\infty}(S^3 \times S^{6} \times S^{8}) \rtimes_\beta \mathbb{Z})\cong \mathbb{C}, & i= 1,3,7,9\\
E_\infty^i(S^3 \times S^{6} \times S^{8}) \rtimes_\beta \mathbb{Z})\cong \{0\},& i=\ else .
	\end{array}
\right.
\]

\end{proof}

In the same reason we have 

\begin{prop}\label{8gra}
	\[H^0_{coeq}(S^3 \times S^{6} \times S^{8}, \alpha 	)=\mathbb{C},  H^8_{coeq}(S^3 \times S^{6} \times S^{8}, \alpha)=\mathbb{C},\]
	\[H^3_{eq}(S^3 \times S^{6} \times S^{8}, \alpha)=\mathbb{C}, H^{11}_{eq}(S^3 \times S^{6} \times S^{8}, \alpha)=\mathbb{C},\]
	with all the other $H^*_{eq}(S^3 \times S^{6} \times S^{8}, \alpha)$ and $H^*_{coeq}(S^3 \times S^{6} \times S^{8}, \alpha)$ nothing but $\{0\}$. The grading structure (deRham homology) are then
	\[
\left\{	
\begin{array}{cc}
E_\infty^i(C^{\infty}(S^3 \times S^{6} \times S^{8}) \rtimes_\alpha \mathbb{Z})\cong \mathbb{C}, & i= 1,3,9,11\\
E_\infty^i(C^{\infty}(S^3 \times S^{6} \times S^{8}) \rtimes_\alpha \mathbb{Z})\cong \{0\},& i=\ else .
	\end{array}
\right.
\]
\end{prop}

Proposition \ref{6gra} and \ref{8gra} implies
\begin{cor}\label{scla}
	The smooth crossed product algebra defined by $\alpha$ is not isomorphism to the one defined by $\beta$, i.e. 
	\[C^\infty (S^3 \times S^{6} \times S^{8})\rtimes_{\alpha}\mathbb{Z} \ncong C^\infty (S^3 \times S^{6} \times S^{8})\rtimes_{\beta}\mathbb{Z}.\]

\end{cor} 

\begin{rem}
As well known, one can not define proper grading structure for $K$-theory of operator algebras (see \cite{liu} for an example constructed by Elliott and Gong). Sometimes the absence of grading structure account for the reason why different diffeomorphisms can give the same $C^*$-algebras. In the meantime, smooth algebras are still different since their cyclic cohomology 	bear different grading structure. 
\end{rem}

\section{ the proof of Theorem \ref{cons}}

The next two propositions are obtained in \cite{fathi}. 

\begin{prop}[corollary 4.11, \cite{fathi} ]\label{pmin}
Let $M$ be a finite dimensional compact manifold, there exists a free $S^1$ action on $M$. Let $U$ be a non-empty open set of $M$. Then there exist a diffeomorphism $H$ on $M$, isotopic to identity, such that $H^{-1}(U)$ cut every orbit of $S^1$ action.
\end{prop}

\begin{prop}[lemma 6.3, \cite{fathi}]\label{perg}
Let $M$ be a finite dimensional compact manifold, there is a free $S^1$ action on $M$. Let $U$ be a non-empty open set of $M$. let $\epsilon$ be a strictly positive real number. Then there is a diffeomorphism $H$ on $M$, isotopy to identity, such that for any point $y\in M $, we have 
\[m\{t\in S^1|\mathcal{R}_t (y) \notin H^{-1}(U) \}<\epsilon,\]
where $m$ is the Lesbegue measure on $S^1$. 
\end{prop}

Now we are ready to prove Theorem \ref{cons}. It is sufficient to prove that there are minimal unique ergodic diffeomorphisms in $\overline{\Omega(\mathcal{R}, \mathcal{P}_6)}$. The $S^1$ free action, $\mathcal{R}$ and $\mathcal{P}_6$ are the one given in the first section. We prove two key lemmas first. In the following, we may represent an element in $S^1$ by its angle $\theta\in (0,1]$.

\begin{lem}\label{lmin}
Given $\frac{p}{q} \in \mathbb{Q}/\mathbb{Z}$, $U$ a non-empty open set in $S^3 \times S^6 \times S^8$. Then there is a diffeomorphism $H$ of $M$ such that:\\
$\mathnormal{1)}$ $H\circ \mathcal{R}_{\frac{p}{q}}\circ H^{-1}={R}_{\frac{p}{q}}, H\circ \mathcal{P}_6 \circ H^{-1}=\mathcal{P}_6$,\\
$\mathnormal{2)}$ For any $y\in S^3 \times S^6 \times S^8$, $H^{-1}(U)$ cut either the orbit under the free $S^1$ action of $y$ or the one of $\mathcal{P}_6(y)$.
\end{lem}
\begin{proof}
	Let $G=\{i\frac{p}{q}, i=1,\dots,q\}$ be the subgroup of $S^1$ generated by $\exp{2\pi i\frac{p}{q}}$, $B={I, \mathcal{P}_6}$ an order 2 group acting on $S^3 \times S^6 \times S^8$. Consider the manifold defined by $\pi: S^3 \times S^6 \times S^8 \rightarrow S^3 \times S^6 \times S^8 / (G\times B) $. As implied by Proposition \ref{pmin} , there is a diffeomorphism $\overline{H}: S^3 \times S^6 \times S^8 / (G\times B) \to S^3 \times S^6 \times S^8 / (G\times B) $, such that $\overline{H}^{-1}(\pi U)$ cut any orbit of $S^1$ on 
	$S^3 \times S^6 \times S^8 / (G\times B) = S^3/G \times \mathbb{R P}_6 \times S^8$. Choose the the lift $H$ of $\overline{H}$ which keep the orientation. One can see $H$ satisfies  conditions $\mathnormal{1)}$ and $\mathnormal{2)}$.
	\end{proof}

\begin{lem}\label{lerg}
Given $\phi \in C(S^3 \times S^6 \times S^8)$, $\eta >0$ and $\frac{p}{q} \in \mathbb{Q}/\mathbb{Z}$, we can choose a diffeomorphism and constant $c\in \mathbb{R}$ such that:\\
$\mathnormal{1)}$ $H\circ \mathcal{R}_{\frac{p}{q}}\circ H^{-1}={R}_{\frac{p}{q}}, H\circ \mathcal{P}_6 \circ H^{-1}=\mathcal{P}_6$,\\
$\mathnormal{2)}$ for all $y\in M$, $|\int_{S^1} (\phi\circ H \circ \mathcal{R}_t(y)+ \phi\circ H \circ \mathcal{R}_t \circ \mathcal{P}_6(y) )dt -c|< \eta$.
\end{lem}
\begin{proof}
	Consider $\pi: S^3 \times S^6 \times S^8 \rightarrow S^3 \times S^6 \times S^8 / G\times B $ as we have done. $\mathcal{R}_t$ naturally defines a diffeomorphism $\overline{\mathcal{R}}_t$ of $S^3 \times S^6 \times S^8 / (G\times B) $ from . For $\phi$, define a function $\overline{\phi} =\frac{1}{q} \sum^{q-1}_{i=0}( \phi\circ \mathcal{R}_{\frac{i}{q}} + \phi \circ \mathcal{P}_6 \circ\mathcal{R}_{\frac{i}{q}})$. For any $x\in S^3 \times S^6 \times S^8$ and sufficiently small $\epsilon$, consider the open set 
	\[U=\{\overline{x}\in S^3 \times S^6 \times S^8 / (G\times B), |\overline{\phi}(\overline{x})|<\epsilon\}.\]
	By Proposition \ref{perg} , there is a diffeomorphism $\overline{H}$ such that 
	\[m\{\theta\in S^1|\mathcal{R}_t \pi(x) \notin \overline{H}^{-1}(U) \}<\epsilon,\] 
	thus we have
	\[|\int_{S^1} \overline{\phi}\circ \overline{H}\circ \overline{\mathcal{R}_t}(\overline{x}) dt|<\epsilon+\epsilon\|\phi\|<\eta.\]
	$\phi$ is the lift of $\overline{\phi}$. $\phi$ satisfies condition $\mathnormal{1)}$ and $\mathnormal{2)}$. 
\end{proof}

As shown in Proposition \ref{mind} and \ref{ergd}, the set of minimal unique ergodic diffeomorphisms $\mathcal{O}$ in $\overline{\Omega(\mathcal{R},\mathcal{P}_6)}$ equals to $ \cap_i \iota(U_i)\cap (\cap_i \cap_{k\geq1} \iota(\phi_i, 1/k )) \cap \overline{\Omega(\mathcal{R},\mathcal{P}_6)}$. We will prove that this $G_\delta$ set is dense in $\overline{\Omega(\mathcal{R},\mathcal{P}_6)}$. It would be sufficient to prove that $\mathcal{R}_{\frac{p}{q}}$ belongs to $\overline{\iota(U_i)}$ and $\overline{\iota(\phi_l,1/k)}$ for any $i$, $l$ and $k$.\\
 $\mathnormal{1)}$ For any $U$, $H$ is a diffeomorphim in Lemma \ref{lmin} . For any $x$, $U$ cut either the set
\[H\circ\mathcal{R}_{S^1}\circ H^{-1} (x)\triangleq \{H\circ\mathcal{R}_t\circ H^{-1} (x), t\in [0,1)\}\]
  or 
\[H\circ\mathcal{R}_{S^1}\circ \mathcal{P}_6 \circ H^{-1} (x)\triangleq \{H\circ\mathcal{R}_t\circ\mathcal{P}_6\circ H^{-1} (x), t\in [0,1)\}.\]
For any irrational number $\theta\in [0,1)$, both sets 
$\{H\circ \mathcal{R}_{2n\theta} \circ H^{-1}(x)\}_{n\in \mathbb{Z}}$ and 
$\{H\circ \mathcal{R}_{(2n+1)\theta} \circ \mathcal{P}_6 \circ H^{-1}(x)\}_{n\in \mathbb{Z}}$ are dense in 
$H\circ\mathcal{R}_{S^1}\circ H^{-1} (x)$ and 
$H\circ\mathcal{R}_{S^1}\circ \mathcal{P}_6 \circ H^{-1} (x)$ respectively. Thus 
$H\circ \mathcal{R}_\theta \circ \mathcal{P}_6 \circ H^{-1}$ belongs to $\iota(U)$ for any irrational number $\theta\in [0,1)$. Since any $\frac{p}{q}$ can be approximated by a sequence of irrational numbers $\theta_n$, there is $H\circ \mathcal{R}_{\theta_n} \circ \mathcal{P}_6 \circ H^{-1}\to H\circ \mathcal{R}_{\frac{p}{q}} \circ \mathcal{P}_6 \circ H^{-1}=\mathcal{R}_{\frac{p}{q}} \circ \mathcal{P}_6$. Thus $\mathcal{R}_{\frac{p}{q}}$ belongs to $\overline{\iota(U)}$.\\
 $\mathnormal{2)}$ Given $\phi$ and $\epsilon>0$, $H$ is a diffeomorphism in Lemma \ref{lerg}. Since irrational rotation transformation is unique ergodic on $S^1$, for any $x$ and any irrational number $\theta$, there is
 \[\lim_{n\to\infty}\frac{1}{n}\sum^{n-1}_{i=0}\phi\circ \mathcal{R}_{i\theta}(x)=\int_{S^1} \phi\circ \mathcal{R}_t(x) dt.\]
 Thus the following holds 
  \[\lim_{n\to\infty}\frac{1}{n}\sum^{n-1}_{i=0}\phi\circ \mathcal{R}_{2i\theta}(x)=\int_{S^1} \phi\circ \mathcal{R}_t(x) dt.\]
\begin{eqnarray*}
\lim_{n\to\infty}\frac{1}{n}\sum^{n-1}_{i=0}\phi\circ \mathcal{R}_{(2i+1)\theta}\circ \mathcal{P}_6 (x)&=&\int_{S^1} \phi\circ \mathcal{R}_\theta \mathcal{R}_t\circ \mathcal{P}_6(x) dt\\
&=& \int_{S^1} \phi\circ \mathcal{R}_t\circ \mathcal{P}_6(x) dt.
\end{eqnarray*}
Hence $H\circ \mathcal{R}_{\theta} \circ \mathcal{P}_6 \circ H^{-1} $ belongs to $\iota(\phi, \epsilon)$. Find a sequence of irrational number $\theta_n$ converge to $\frac{p}{q}$. Since  $H\circ \mathcal{R}_{\theta_n} \circ \mathcal{P}_6 \circ H^{-1}\to H\circ \mathcal{R}_{\frac{p}{q}} \circ \mathcal{P}_6 \circ H^{-1}=\mathcal{R}_{\frac{p}{q}} \circ \mathcal{P}_6$, we have $\mathcal{R}_{\frac{p}{q}}\subset \overline{\iota(\phi, \epsilon )}$.

Now we have completed the proof of Theorem \ref{cons}.

\textit{Acknowledgement}

\end{document}